\documentclass[12pt]{amsart}

\usepackage{amsfonts}
\usepackage{latexsym}
\usepackage{amsmath}
\usepackage{amssymb}
\usepackage{amsthm}
\usepackage{enumerate} 
\usepackage{mathtools} 
\usepackage[margin=1.2in]{geometry}
\usepackage{mathabx}

\theoremstyle{plain}
\newtheorem{theorem}{Theorem}
\newtheorem{lemma}{Lemma}
\newtheorem{proposition}{Proposition}

\theoremstyle{definition}
\newtheorem{definition}{Definition}
\newtheorem{remark}{Remark}

\newcommand\naturals{\mathbb{N}}
\newcommand\integers{\mathbb{Z}}

\newcommand\reals{\mathbb{R}}
\newcommand\complexes{\mathbb{C}}

\newcommand\norm[1]{\left\lVert#1\right\rVert}
\newcommand\ev[2]{\left<#1,#2\right>}

\DeclareMathOperator{\supp}{supp}

\newcommand\bigoh[1]{O\left(#1\right)}

\newcommand\ds{\mathcal{D}}

\newcommand\es{\mathcal{E}}

\newcommand\des{\es^{\prime}}

\newcommand\despaceomega[1]{\testspaceomega{\des}}

\newcommand\beurou{\ast}
\newcommand\utestspace[3]{#1^{#2}(#3)}
\newcommand\utestspacereals[3]{\utestspace{#1}{#2}{\reals^{#3}}}

\newcommand\utestspacereal[2]{#1^{#2}}

\newcommand\udspacereals[2]{\utestspacereals{\ds}{#1}{#2}}

\newcommand\uespacereals[2]{\utestspacereals{\es}{#1}{#2}}

\newcommand\uddspacereals[2]{\utestspacereals{\ds}{\prime #1}{#2}}
\newcommand\uddspacereal[1]{\utestspacereal{\ds}{\prime #1}}

\newcommand\udespacereals[2]{\utestspacereals{\es}{\prime #1}{#2}}

\begin{document}

\title[Asymptotic boundedness and moment asymptotic expansion]{Asymptotic boundedness and moment asymptotic expansion in ultradistribution spaces}

\author[L. Neyt]{Lenny Neyt}
\thanks{L. Neyt gratefully acknowledges support by Ghent University, through the BOF-grant 01J11615.}

\author[J. Vindas]{Jasson Vindas}
\thanks {J. Vindas was supported by Ghent University through the BOF-grants 01J11615 and 01J04017.}

\address{Department of Mathematics: Analysis, Logic and Discrete Mathematics\\ Ghent University\\ Krijgslaan 281\\ 9000 Gent\\ Belgium}
\email{lenny.neyt@UGent.be}
\email{jasson.vindas@UGent.be}

\dedicatory{Dedicated to the memory of Prof. Bogoljub Stankovi\'{c}}

\subjclass[2010]{46F05; 46F10}

 \keywords{Asymptotic behavior of generalized functions; moment asymptotic expansions; S-asymptotics; quasiasymptotics; ultradistribution spaces}

\begin{abstract}
	We obtain structural theorems for the so-called S-asymptotic and quasiasymptotic boundedness of ultradistributions. Using these results, we then analyze the moment asymptotic expansion (MAE), providing a full characterization of those ultradistributions satisfying this asymptotic formula in the one-dimensional case. We also introduce and study a uniform variant of the MAE. \end{abstract}

\maketitle

\section{Introduction}

In this article we study asymptotic properties of ultradistributions. Asymptotic analysis is an important subject in the theory of generalized functions and provides powerful tools for applications in areas such as mathematical physics, number theory, and differential equations. The theory of asymptotic behavior of generalized functions has been particularly useful in the study of Tauberian theorems for integral transforms.   We refer to the monographs \cite{estrada2012distributional, stevan2011asymptotic, vladimirov1988tauberian} for complete accounts on the subject and its many applications. 

The asymptotic behavior of a generalized function is usually analyzed via its parametric behavior, mostly with respect to translation or dilation. The idea of looking at the translates of a distribution goes back to Schwartz \cite[Chapter VII]{Schwartz}, who used it to measure the order of growth of tempered distributions at infinity. Pilipovi\'{c} and Stankovi\'{c} later introduced a generalization, the so called \emph{S-asymptotic} behavior, and thoroughly investigated its properties for distributions, ultradistributions, and Fourier hyperfunctions. There are deep connections between S-asymptotics and Wiener Tauberian theorems for generalized functions \cite{pilipovic-stankovic94}. In Section \ref{Section S-asymptotic boundedness} of this article we study the structure of S-asymptotically bounded ultradistributions; in fact, we discuss a counterpart of the main structural result from \cite{PropSasymp} (cf. \cite[Theorem~1.10, p.~46]{stevan2011asymptotic}) for S-asymptotic boundedness, which we prove to hold under much weaker assumptions than those employed there. 

There are two very well-established approaches to asymptotics of generalized functions related to dilation. The first one is the \textit{quasiasymptotic} behavior, which employs regularly varying functions \cite{bingham1989regular} as gauges in the asymptotic comparisons. This concept was first introduced by Zav'yalov for Schwartz distributions \cite{Zavyalov} and further developed by him, Drozhzhinov, and Vladimirov in \cite{vladimirov1988tauberian}. The quasiasymptotic behavior was naturally extended to the context of one-dimensional ultradistributions in \cite{10.2307/44095755}. Here we shall provide complete structural theorems for quasiasymptotically bounded ultradistributions in Section \ref{sec:quasiasympbound}, both at infinity and the origin, using the above-mentioned results for S-asymptotic boundedness and techniques developed by the authors in \cite{structquasiultra}.  These results are the ultradistributional analogs of the structural theorems from \cite{VindStrucQbounded} for Schwartz distributions (cf. \cite[Section~2.12, p.~160]{stevan2011asymptotic}). 

The second important approach to asymptotic behavior related to dilation is the so-called \textit{moment asymptotic expansion} (MAE), whose properties have been extensively investigated by Estrada and Kanwal \cite{E-K-AsympExpansions, estrada2012distributional}. Some recent contributions can be found in \cite{Schmidt2005, yang-estrada}. A generalized function $f$ is said to satisfy the MAE if there is a certain multisequence $\{\mu_{\alpha}\}_{\alpha \in \naturals^{d}}$, called the moments of $f$, such that the following asymptotic expansion holds
	\begin{equation}
		\label{eq:MAEasymp}
		f(\lambda x) \sim \sum_{\alpha \in \naturals^{d}} \frac{(-1)^{|\alpha|} \mu_{\alpha} \delta^{(\alpha)}(x)}{\alpha! \lambda^{|\alpha| + d}} , \qquad \lambda \rightarrow \infty . 
	\end{equation}
As is shown in the monograph \cite{estrada2012distributional}, the MAE supplies a unified approach to several aspects of asymptotic analysis and its applications. Interestingly, Estrada characterized \cite{Estrada 1998} the largest space of distributions where the MAE holds as the dual of the space of so-called GLS symbols \cite{GLS1968}. We will consider the MAE for ultradistributions, providing in Section \ref{Section MAE}  a counterpart of Estrada's full characterization in the one-dimensional case. We shall also study a uniform version of (\ref{eq:MAEasymp}) in Section \ref{section UMAE}, which we call the UMAE. Our considerations naturally lead to introduce the ultradistribution spaces $\mathcal{K}'^{\ast}(\mathbb{R}^{d})$ and $\mathcal{K}'^{\ast}_{\dagger}(\mathbb{R}^{d})$, which are intimately connected with the MAE and UMAE. We note that in even dimension  our space $\mathcal{K}'^{\ast}_{\dagger}(\mathbb{R}^{2d})$ arises as the dual of one of the spaces of symbols of `infinite order' pseudo-differential operators from \cite{BojanPseudo}.

\section{Preliminaries}\label{section preliminaries}

In this section we fix the notation and briefly collect some background material on ultradistributions. Given a weight sequence $\{M_{p}\}_{p \in \naturals}$ of positive real numbers, we will often make use of one or more of the following conditions:
	\begin{description}
		\item[$(M.1)$] $M_{p}^{2} \leq M_{p - 1} M_{p + 1}$, $p \geq 1$; 
		
		\item[$(M.2)'$] $M_{p + 1} \leq A H^{p} M_{p}$, $p \in \naturals$, for certain constants $A$ and $H$;
		\item[$(M.2)$] $M_{p+q} \leq A H^{p+q} M_{p} M_{q}$, $p,q \in \naturals$, for certain constants $A$ and $H$;
	        \item[$(M.3)'$] $\sum_{p = 1}^{\infty} M_{p-1}/M_p < \infty$;
		\item[$(M.3)$] $ \sum_{p=q}^{\infty} M_{p-1}/M_p \leq c_0 q M_{q-1}/M_q$, $q\geq 1$, for a certain constant $c_0$.
	\end{description}
The meaning of all these conditions is very well explained in \cite{ultradistributions1}. 
 For multi-indices $\alpha \in \naturals^{d}$, we will simply denote $M_{|\alpha|}$ as $M_{\alpha}$.

Let $\Omega\subset \mathbb{R}^{d}$ be open and let $K \subset \reals^{d}$ be a (regular) compact subset. For any $\ell > 0$ we denote as $\mathcal{E}^{M_{p}, \ell}(K)$ the space of all smooth functions $\varphi$ for which 
$$\norm{\varphi}_{\mathcal{E}^{M_{p}, \ell}(K)} := \sup_{\alpha \in \naturals^{d}} \sup_{x \in K} \frac{ |\varphi^{(\alpha)}(x)| }{\ell^{|\alpha|} M_{\alpha}}$$
is finite. Then, as customary \cite{ultradistributions1}, we set
$$
 \mathcal{E}^{(M_p)}(\Omega)=\varprojlim_{K \Subset \Omega} \varprojlim_{\ell>0}\mathcal{E}^{M_{p}, \ell}(K) \qquad \mbox{and} \qquad \mathcal{E}^{\{M_p\}}(\Omega)= \varprojlim_{K \Subset \Omega} \varinjlim_{\ell>0}\mathcal{E}^{M_{p}, \ell}(K) ,
$$
and use  $\mathcal{E}^{\beurou}(\Omega)$ as the common notation for the Beurling and Roumieu case (when a separate treatment is needed, we will always first state assertions about the Beurling case $(M_{p})$, then followed by the Roumieu case $\{M_{p}\}$ in parenthesis). 

We shall make use of Komatsu's well-known projective limit description of $\mathcal{E}^{\{M_p\}}(\Omega)$. Consider the directed set 
		$\mathfrak{R} = \left\{ (\ell_{p})_{p \in \integers_{+}} : \ell_{p} \nearrow \infty \text{ and } \ell_{p} > 0, \forall p \in \integers_{+} \right\} .$ For any $(\ell_{p}) \in \mathfrak{R}$ we write $L_{0} = 1$ and $L_{p} = \prod_{j = 1}^{p} \ell_{j}$ for $p \in \integers_{+}$. Then, it is shown in \cite[Proposition 3.5, p.~675]{ultradistributions3} that if the weight sequence satisfies $(M.1)$, $(M.2)'$, and $(M.3)'$, then
\begin{equation}
\label{project limit E}
\mathcal{E}^{\{M_p\}}(\Omega)= \varprojlim_{(\ell_{p}) \in \mathfrak{R}} \mathcal{E}^{(L_{p}M_{p})}(\Omega)
\end{equation}		
as locally convex spaces.

For compactly supported test functions our notation is standard, we write $\mathcal{D}^{\ast} (\Omega)=\mathcal{E}^{\ast} (\Omega)\cap\mathcal{D} (\Omega)$ and $\mathcal{D}^{M_p,\ell}_{K}=\mathcal{E}^{M_p,\ell}(K)\cap \mathcal{D}_{K}$ with $K\subset \Omega$ compact, these spaces are topologized in the canonical way \cite{ultradistributions1}. (Naturally, by the Denjoy-Carleman theorem and under $(M.1)$, the non-triviality of $\mathcal{D}^{\ast} (\Omega)$ is equivalent to $(M.3)'$.) 
The strong dual $\mathcal{D}'^{\ast} (\Omega)$ is the space of ultradistributions of class $\ast$, while the elements of $\mathcal{E}'^{\ast} (\Omega)$ are exactly those ultradistributions with compact support.

\section{S-asymptotic boundedness}\label{Section S-asymptotic boundedness}
 Our aim in this section is to study the structure of those ultradistributions that satisfy 
\begin{equation}
\label{eq:Sasympboundcone}
				f(x + h) = \bigoh{\omega(h)}, \quad  h \in W , \quad \mbox{in } \mathcal{D}'^{\ast}(\mathbb{R}^{d}),
			\end{equation}
where $W\subseteq \mathbb{R}^{d}$ is simply an unbounded set  and $\omega$ is a positive function. The S-asymptotic relation is to be interpreted in the ultradistributional sense \cite{stevan2011asymptotic}, that is, it explicitly means that for each test function $\varphi\in\mathcal{D}^{\ast}(\mathbb{R}^{d}),$
\begin{equation}
\label{eq:Sasympboundconetest}
\sup_{h\in W}\frac{\ev{f(x+h)}{\varphi(x)}}{\omega(h)}=\sup_{h\in W}\frac{(f\ast \check{\varphi})(h)}{\omega(h)}<\infty.
\end{equation}

We begin with a simple auxiliary lemma that allows us to preserve certain growth properties when regularizing functions. Given $R>0$ and a set $W$, we denote as $W_R$ the open $R$-neighborhood of $W$, that is, the set $W_R=W + B(0,R)$.

\begin{lemma}\label{lemma regularizing L^infinity}  Given $R>0$ there are absolute constants $c_{0,R}$ and $c_{1,R}$ such that each function $g\in L^{\infty}_{loc}(W_{R})$ satisfying the bound $\sup_{x\in W,\: |h|< R} |g(x+h)|/\omega(x)<\infty$, where $W\subset\mathbb{R}^{d}$ and $\omega$ is a positive function defined on $W$, can be written as  $g=\Delta g_{1}+g_0$ in $W_R$ for some functions $g_{j}\in C(\mathbb{R}^{d})$ that satisfy
$$
\sup_{x\in W} \frac{|g_{j}(x)|}{\omega(x)}\leq c_{j,R}\sup_{x\in W,\ |h|<R} \frac{|g(x+h)|}{\omega(x)}, \qquad j=0,1.
$$

\end{lemma}
\begin{proof}
To show this,  we make use of the fact that the fundamental solutions of the Laplacian belong to $L^{1}_{loc}(\mathbb{R}^{d})\cap C^{\infty}(\mathbb{R}^{d}\setminus\{0\})$. By cutting-off a fundamental solution in the ball $B(0,R)$, this implies we can select functions $\chi_1\in L^{1}(\mathbb{R}^{d})$ and $\chi_0\in \mathcal{D}(\mathbb{R}^{d})$ both supported on $B(0,R)$, such that $\delta=\Delta \chi_1+\chi_0$. Extend $g$ off $W_{R}$ as 0 and keep calling this extension by $g$. We obtain the claim if we set $g_j=g\ast\chi_{j}$ so that the desired inequalities hold with $c_{j,R}=\int_{|x|\leq R} |\chi_{j}(-x)|dx.$

\end{proof}

 We can now obtain our first structural theorem for \eqref{eq:Sasympboundcone} in the case where the weight sequence satisfies mild hypotheses. For it, we shall need to impose the following regularity condition on the gauge function $\omega$, 

\begin{equation}
		\label{eq:omegacond}
		\sup_{x\in\mathbb{R}^{d}} \frac{\omega(\: \cdot\: +x)}{\omega(x)}\in L^{\infty}_{loc}(\mathbb{R}^{d}).
	\end{equation}
 In the proof of the next theorem we employ a technique of G\'{o}mez-Collado that she applied to obtain various characterizations of the space of bounded ultradistributions in  \cite{Gomez-Collado}.
	
	\begin{theorem}\label{structural th S-asymptotic}
	
		Let $W\subset \mathbb{R}^{d}$ be an unbounded set and let $\omega$ be a positive measurable function on $\mathbb{R}^{d}$ that satisfies \eqref{eq:omegacond}. Suppose $(M.1)$, $(M.2)'$, and $(M.3)'$ hold. Then, an ultradistribution $f \in \uddspacereals{\beurou}{d}$ satisfies \eqref{eq:Sasympboundcone} if and only if for each $R>0$ there are continuous functions $\{f_{\alpha}\}_{\alpha\in\mathbb{N}^{d}}$ defined on $W_{R}$ such that for some $\ell>0$ (for each $\ell>0$) there exists $ C_{\ell} > 0$ for which the bounds	
			\begin{equation}
				\label{eq:Bstructbounds}
				\left|f_{\alpha}(x)\right| \leq C_{\ell} \frac{\ell^{|\alpha|}}{M_{\alpha}} \omega(x) , \qquad  x \in W_{R},\ \alpha \in \naturals^{d} , 
			\end{equation}
 hold and
			\begin{equation}
				\label{eq:Bstruct}
				f = \sum_{\alpha \in \naturals^{d}} f^{(\alpha)}_{\alpha} \qquad \mbox{ in } W_R.
			\end{equation}		
	\end{theorem}

	\begin{proof}
		The sufficiency of the conditions is easily verified. For the necessity, we start by making some reductions. In view of Lemma \ref{lemma regularizing L^infinity} and the assumption $(M.2)'$, it suffices to establish \eqref{eq:Bstruct} with \eqref{eq:Bstructbounds} for functions that are merely measurable. Next, we show that we may assume that $W=\mathbb{R}^{d}$.
		Let $R > 0$ be arbitrary and let $\chi_{R}$ be a non-negative smooth function on $\reals^{d}$ such that $0 \leq \chi_{R} \leq 1$, $\chi_{R} = 1$ on $W_{R}$ while $\chi_{R} = 0$ outside $W_{2R}$, and such that
			\[ \sup_{\alpha \in \naturals^{d},\: \xi \in \reals^{d}} \frac{|\chi_{R}^{(\alpha)}(\xi)|}  {\ell^{|\alpha|}   M_{\alpha}} < \infty , \qquad \forall \ell > 0 . \]
	Then, we set $\widetilde{f} := \chi_{R} \cdot f$ and notice that $\widetilde{f}$ and $f$ coincide on $W_{R}$. Take any $\varphi \in \udspacereals{\beurou}{d}$ and let $r > 0$ be such that $\supp \varphi \subset B(0, r)$. Take any $h \in \reals^{d}$. If $h \notin W_{r + 2 R}$ then $\ev{\widetilde{f}(x + h) / \widetilde{\omega}(h)}{\varphi(x)} = 0$. Suppose now $h \in W_{r + 2 R}$, then $h = h_{1} + h_{2}$ with $h_{1} \in W$ and $h_{2} \in B(0, r + 2 R)$. Then, employing \eqref{eq:omegacond} and the Banach-Steinhaus theorem,
				\[ \frac{|\ev{\widetilde{f}(x + h)}{\varphi(x)}|}{\widetilde{\omega}(h)} \leq C_{R + 2\varepsilon} \frac{|\ev{\widetilde{f}(x + h_{1})}{\varphi(x - h_{2})}|}{\omega(h_{1})} = \bigoh{1} , \]
 because $\{T_{h_{2}} \varphi : h_{2} \in B(0, r + 2 R)\}$ is a bounded family in $\udspacereals{\beurou}{d}$. Consequently, $\widetilde{f}(x+h)=O(\omega(h))$, $h\in\mathbb{R}^{d}$, in $\mathcal{D}'^{\ast}(\mathbb{R}^{d})$
 
 We may therefore assume w.l.o.g. that $\tilde{f}=f$ and derive \eqref{eq:Bstructbounds} and \eqref{eq:Bstruct} on the whole $\mathbb{R}^{d}$ for measurable functions $f_{\alpha}$  under the hypothesis that $\{f(x+h)/\omega(h):\: h\in\mathbb{R}^{d}\}$ is a bounded subset of $\mathcal{D}'^{\ast}(\mathbb{R}^{d})$. We now reason as in \cite{Gomez-Collado}. Let $\psi\in\mathcal{D}^{\ast}_{[-1,1]^{d}}$ be such that $\sum_{n\in\mathbb{Z}^{d}}\psi(x-n)=1$ for each $x\in\mathbb{R}^{d}$. We have that $\{\psi f(\:\cdot\:+n)/\omega(n):\: n\in\mathbb{Z}^{d}\}$ is now a bounded set in the space $\mathcal{E}'^{\ast}(\mathbb{R}^{d})$. Using $(M.2)'$, we obtain in the Beurling case the existence of $\ell>0$ and in the Roumieu case of $(\ell_{p})\in\mathfrak{R}$ such that for some $C>0$ and all $n\in\mathbb{Z}^{d}$ and $\phi\in\mathcal{E}^{\ast}(\mathbb{R}^{d})$
 \begin{equation}
 \label{eq:auxStrucSB}
| \langle f(x), \psi(x-n)\phi(x-n) \rangle| \leq C \sum_{\alpha\in\mathbb{N}^{d}} \frac{\omega(n)}{L_{\alpha}M_\alpha}\int_{[-1,1]^{d}} |\phi^{(\alpha)}(x)|dx,
 \end{equation}
 where in the Roumieu case we have used the projective description \eqref{project limit E} and we have set $L_{p}=\ell^{p}$ in the Beurling case and $L_{p}=\prod_{j\leq p}\ell_{p}$ in the Roumieu one. We further on consider the Banach space $X$ of all $\varphi\in C^{\infty}(\mathbb{R}^d)$ such that
  \[\norm{\varphi}_{X}=\sum_{\alpha\in\mathbb{N}^{d}}\int_{\mathbb{R}^{d}} |\varphi^{(\alpha)}(x)| \frac{ \omega(x)}{L_{\alpha}M_{\alpha}}dx <\infty.\] 
 Let $\varphi\in\mathcal{D}^{\ast}(\mathbb{R}^{d})$ be arbitrary. Applying \eqref{eq:auxStrucSB} to each  $\phi(x)=\varphi(x+n)$ and using the hypothesis \eqref{eq:omegacond}, we obtain, with $C'=C\sup_{x\in\mathbb{R}^{d}, \ y\in [-1,1]^{d}} \omega(x+y)/\omega(x)$,
\begin{align*}
|  \langle f, \varphi \rangle|& \leq \sum_{n\in\mathbb{Z}^{d}}\left|\langle f(x), \psi(x-n)\varphi(x) \rangle\right|\leq  3^{d}C'  \norm{\varphi}_{X}.
\end{align*}
Note that $(M.1)$ and $(M.3)'$ ensure that $\mathcal{D}^{\ast}(\mathbb{R}^{d})$ is dense in $X$ and hence we conclude $f\in X'$. Embedding $X$ into $L^{1}(\mathbb{N}^{d}\times \mathbb{R}^{d}, d\mu)$ via the isometry $j(\varphi)(\alpha,x)= (-1)^{|\alpha|}\varphi^{(\alpha)}(x)$, where the measure is given by $d\mu= \omega(x)/(L_{\alpha}M_{\alpha}) d\alpha dx$ with $d\alpha$ the natural counting measure on $\mathbb{N}^{d}$, we can apply the Hanh-Banach theorem to get the representation \eqref{eq:Bstruct} with measurable functions $f_{\alpha}$ on $\mathbb{R}^{d}$ that satisfy bounds $|f_{\alpha}(x)|\leq C'' \omega(x)/(L_{\alpha}M_{\alpha})$. This yields already the result in the Beurling case. In the Roumieu case we finally employ \cite[Lemma~3.4(ii), p.~674]{ultradistributions3} to obtain the bounds
 \eqref{eq:Bstructbounds} for each $\ell>0$ and some $C_{\ell}>0$.

	\end{proof}

In applications it is very useful to combine Theorem \ref{structural th S-asymptotic} with the ensuing proposition, which provides conditions under which one might essentially apply Theorem \ref{structural th S-asymptotic} with a function $\omega$ that is just defined on the set $W$. 

	\begin{proposition}
		\label{p:extendweightfunc}
		Let $W\subset \mathbb{R}^{d}$ be a closed convex set. Any positive function $\omega$ on $W$ satisfying
			\begin{equation}
				\label{eq:omegacondgeneral}
				(\forall R > 0) \quad \underset{|h|\leq R}{\sup_{x,\: x+h\in W}} \frac{\omega(x+h)}{\omega(x)}<\infty 
						\end{equation}
can be extended to a positive function  on $\reals^{d}$ satisfying \eqref{eq:omegacond}. In addition, if $\omega$ is measurable (or continuous), the extension can be chosen measurable (or continuous) as well.
	\end{proposition}
	\begin{proof}
		For any $x \in \reals^{d}$ we denote by $\widetilde{x} \in W$ the (unique in view of convexity) closest point to $x$ in $W$. Then, we set $\widetilde{\omega}(x) := \omega(\widetilde{x})$. Since $x\mapsto \widetilde{x}$ is continuous, $\widetilde{\omega}$ inherits measurability or continuity if $\omega$ has the property. We now verify \eqref{eq:omegacond} for $\widetilde{\omega}$. Let $R > 0$ and let $C_{R}$ an upper bound for $\omega(t+y)/\omega(y)$, where $y,t+y\in W$ and $t\in\overline{B}(0, R)$. Let $x \in \reals^{d}$ and $h \in \overline{B}(0, R)$ be arbitrary. Consider the points $x, x + h, \widetilde{x}$ and $\widetilde{x + h}$. By the obtuse angle criterion, the angles defined by the line segments $[x, \widetilde{x}, \widetilde{x + h}]$ and $[x + h, \widetilde{x + h}, \widetilde{x}]$ are at least $\pi/2$, whence $|\widetilde{x} - \widetilde{x + h}| \leq |x - (x + h)| \leq R$. It then follows that $\widetilde{\omega}(x + h) \leq C_{R} \widetilde{\omega}(x)$, as required.
	\end{proof}

If the weight sequence satisfies stronger assumption, one can drop any regularity assumption on $\omega$, as stated in the next result.

\begin{theorem}\label{structural th S-asymptotic, (M.3)}
	
		Let $W\subset \mathbb{R}^{d}$ be an unbounded set and let $\omega$ be a positive function on $W$. Suppose that $(M.1)$, $(M.2)$, and $(M.3)$ hold. An ultradistribution $f \in \uddspacereals{\beurou}{d}$ satisfies \eqref{eq:Sasympboundcone} if and only if for each $R>0$ there are continuous functions $\{f_{\alpha}\}_{\alpha\in\mathbb{N}^{d}}$ defined on $W_{R}$ such that for some $\ell>0$ (for each $\ell>0$) there
exists $ C_{\ell} > 0$ such that
			\begin{equation}
				\label{eq:Bstructboundsibis}
				\left|f_{\alpha}(x+h)\right| \leq C_{\ell} \frac{\ell^{|\alpha|}}{M_{\alpha}} \omega(x) , \qquad  x \in W,\ |h|<R,\ \alpha \in \naturals^{d} ,
				\end{equation}		
 and the representation \eqref{eq:Bstruct} holds. 
		
	\end{theorem}
\begin{proof}
The proof is similar to that of \cite[Theorem 1.10, p. 46]{stevan2011asymptotic}, but we provide some simplifications. The converse is easy to show, so we concentrate on showing the necessity of the conditions for the S-asymptotic boundedness relation \eqref{eq:Sasympboundcone}. Let $R>0$. We consider the linear mapping $A:\mathcal{D}^{\ast}(\mathbb{R}^{d})\to X$, with values in the Banach space $X=\{g:W\to \mathbb{C}:\: \sup_{x\in W} |g(x)|/\omega(x)<\infty\}$, given by $A\varphi= f\ast\varphi.$ It follows from the closed graph theorem that $A$ is continuous.  Consequently, we obtain from the Banach-Steinhaus theorem the existence of $\ell>0$ in the Beurling case or $(\ell_p)\in\mathfrak{R}$ in the Roumieu case such that $A\in(\mathcal{D}_{\overline{B}(0,2R)}^{M_pL_p,1})'$ and $f\ast \varphi \in X$ for each $\varphi \in \mathcal{D}_{\overline{B}(0,2R)}^{M_pL_p,1}$, where we set $L_p=\ell^{p}$ in the Beurling case or $L_p=\prod_{j=1}^{p}\ell_p$ in the Roumieu case. Since for each $\varphi \in \mathcal{D}_{\overline{B}(0,R)}^{M_pL_p,1}$ the set $\{T_{x}\varphi:\: |x|\leq R\}$ is compact in $\mathcal{D}_{\overline{B}(0,2R)}^{M_pL_p,1}$, we conclude that for any such a $\varphi$ the function $f\ast \varphi$ is continuous on $W_{R}$ and
$$
\sup_{h\in W,\: |x|< R}\frac{(f\ast\varphi)(x+h)}{(\omega(h))}<\infty.
$$
We now employ the parametrix method. As shown in \cite[p.~199]{Komatsu1989}, there is an ultradifferential operator $P(D)$ of class $\ast$ that admits a $\mathcal{D}_{\bar{B}(0,R)}^{M_pL_p,1}$-parametrix, namely, for which there are $\chi\in\mathcal{D}^{\ast}_{\bar{B}(0,R)}$ and $\varphi\in\mathcal{D}_{\bar{B}(0,R)}^{M_pL_p,1}$ such that $\delta=P(D)\varphi+ \chi$. Setting $f_0=f\ast\chi$ and $g=f\ast \varphi$, we obtain the decomposition 
$f=P(D)g+f_0$, which in particular establishes the representation \eqref{eq:Bstruct} with functions $f_\alpha$ satisfying the bounds \eqref{eq:Bstructboundsibis}.
\end{proof}

\section{Quasiasymptotic boundedness}

\label{sec:quasiasympbound}

Our results from the previous section can be applied to obtain structural theorems for ultradistributions being quasiasymptotically bounded in dimension 1. Let $\rho$ be a positive function defined on an interval of the form $[\lambda_{0},\infty)$. We are interested in the relation $
f(\lambda x)= O(\rho(\lambda))$ as $\lambda\to\infty$
in ultradistribution spaces. The analog of the condition \eqref{eq:omegacondgeneral} for a function $\rho$ in this multiplicative setting is being $O$-regularly varying (at infinity) \cite[p.~65]{bingham1989regular}. The latter means (cf. \cite[Theorem 2.0.4, p.~64]{bingham1989regular}) that $\rho$ is measurable and for each $R>1$
\[
\limsup_{x\to\infty} \sup_{\lambda\in [R^{-1}, R]}\frac{\rho(\lambda x)}{\rho(\lambda)}<\infty.
\]

 The next proposition can be established with the aid of Theorem \ref{structural th S-asymptotic} and Theorem \ref{structural th S-asymptotic, (M.3)} via an exponential change of variables as in the authors' work \cite[Lemma~3.4]{structquasiultra}; we omit details. 
	\begin{proposition}
		\label{p:quasiasympboundimposesstruct}
		Let $f\in\mathcal{D}^{\prime \beurou}(\mathbb{R})$ and $\rho$ be a positive function. Suppose that $
f(\lambda x)= O(\rho(\lambda))$ as $\lambda\to\infty $ in $\mathcal{D}'^{\ast}(\mathbb{R}\setminus\{0\})$.
\begin{enumerate}[(i)] 
\item If $(M.1)$, $(M.2)'$, and $(M.3)'$ hold and $\rho$ is $O$-regularly varying at infinity,   
then there are continuous functions $f_{m}$ and $x_0>0$ such that
			\begin{equation}\label{f eqstruct}
				f = \sum_{m=0}^{\infty} f_{m}^{(m)} \qquad \mbox{on } \mathbb{R}\setminus[-x_0,x_0] 
			\end{equation}
		and for some $\ell > 0$ (for any $\ell > 0$) there is $C_{\ell} > 0$ such that
			\begin{equation}\label{f eqstruct 1}
				\left| f_{m}(x) \right| \leq C_{\ell} \frac{\ell^{m}}{M_{m}} |x|^{m} \rho(|x|) , \qquad |x|>x_0, \ m\in\mathbb{N}. 
			\end{equation} 

\item  If $(M.1)$, $(M.2)$, and $(M.3)$ hold, for each $R>1$ one can find $x_0$ and continuous functions such that $f$ has the representation \eqref{f eqstruct}, where the $f_m$ satisfy the bounds
\begin{equation}\label{f eqstruct 2}
				\left| f_{m}(a x) \right| \leq C_{\ell} \frac{\ell^{m}}{M_{m}} |x|^{m} \rho(|x|) , \qquad |x|>x_0, \ R^{-1}<a<R,\ m\in\mathbb{N},
			\end{equation}
 for some $\ell > 0$ (for any $\ell > 0$).
\end{enumerate}
	\end{proposition}
\begin{remark}
\label{rk struct quasiasymtotic 1}
Clearly, \eqref{f eqstruct 1} implies \eqref{f eqstruct 2} for an $O$-regularly varying function $\rho$. Assume $(M.1)$, $(M.2)'$, and $(M.3)'$ hold. Notice the representations \eqref{f eqstruct} with bounds \eqref{f eqstruct 2} are also sufficient to yield $
f(\lambda x)= O(\rho(\lambda))$ as $\lambda\to\infty $ in $\mathcal{D}'^{\ast}(\mathbb{R}\setminus\{0\})$, so that the converses of both parts (i) and (ii) of Proposition \ref{p:quasiasympboundimposesstruct} are valid.
\end{remark}

In the rest of the section we are interested in describing quasiasymptotic boundedness in the full space $\mathcal{D}'^{\ast}(\mathbb{R})$. For it, we need to impose  stronger variation  assumptions on the gauge function $\rho$. We call a positive measurable function $O$-slowly varying at infintiy if for each $\varepsilon>0$ there are $C_{\varepsilon},c_{\varepsilon},R_{\varepsilon}>0$ such that
\begin{equation}
\label{eq O-slowly varying}
\frac{c_{\varepsilon}}{\lambda^{\varepsilon}} \leq \frac{L(\lambda x)}{L(x)}\leq C_{\varepsilon}\lambda^{\varepsilon}, \qquad \lambda\geq 1, \ x>R_{\varepsilon}.
\end{equation}
In the terminology from \cite{bingham1989regular} this means that the upper and lower Matuszewska indices of $L$ are both equal to 0. Thus, a function of the form $\rho(\lambda)=\lambda^{q}L(\lambda)$ is an $O$-regularly varying function with both upper and lower Matuszewska indices equal to $q\in\mathbb{R}$.

The authors have found  in \cite{structquasiultra} complete structural theorems for the quasiasymptotic behavior of ultradistributions with respect to regularly varying functions. If we exchange \cite[Lemma 3.4]{structquasiultra} with Proposition \ref{p:quasiasympboundimposesstruct}, the same technique\footnote{One still needs an $O$-version of the integration lemma \cite[Lemma 3.2]{structquasiultra}; however, careful inspection in the arguments given in \cite[Subsection 2.10.2 and Proposition 2.17]{stevan2011asymptotic} shows that having the inequalities \eqref{eq O-slowly varying} is all one needs to establish the validity of such an $O$-version.} from \cite[Sections 3 and 4]{structquasiultra} leads to two ensuing structural theorems for quasiasymptotic boundedness.
	
	\begin{theorem}
		\label{t:quasiasympboundinfinity} Assume $(M.1)$, $(M.2)'$, and $(M.3)'$ hold.
		Let $f \in \uddspacereal{\beurou}(\mathbb{R})$, $q\in\mathbb{R}$, and let $L\in L^{\infty}_{loc}[0,\infty)$ be $O$-slowly varying at infinity. Let $k$ be the smallest positive integer such that $-k \leq q$. Then, 
			\begin{equation} 
				\label{eq:quasiasympboundinfinity}
				f(\lambda x) = \bigoh{\lambda^{q} L(\lambda)} \qquad \text{ as } \lambda \rightarrow \infty \text{ in } \uddspacereal{\beurou}(\mathbb{R}) 
			\end{equation}
holds 
		if and only if there are continuous functions $f_{m}$ on $\reals$ such that
			\begin{equation*}
				f = \sum_{m = k-1}^{\infty} f_{m}^{(m)} , 
			\end{equation*}
		for some $\ell > 0$ (for any $\ell > 0$) there exists $C_{\ell} > 0$ such that
			\begin{equation}
			\label{eq bdd struct quasi infinity}
				\left| f_{m}(x) \right| \leq C_{\ell} \frac{\ell^{m}}{M_{m}} (1 + |x|)^{q + m} L(|x|) , \qquad m \geq k - 1 ,
			\end{equation}
		and additionally (only) when $q=-k$
			\begin{equation}\label{eq alpha=-k struct quasi infinity}
				\int_{-x}^{x} f_{k - 1}(x) dx = \bigoh{L(x)} , \qquad x \rightarrow \infty . 
			\end{equation}
	\end{theorem}

A function $L$ is $O$-regularly varying at the origin if $L(1/x)$ is  $O$-regularly varying at infinity.
	\begin{theorem} Assume $(M.1)$, $(M.2)'$, and $(M.3)'$.
Let $f \in \uddspacereal{\beurou}(\mathbb{R})$, $q\in\mathbb{R}$, and let $L$ be $O$-slowly varying at the origin. Let $k$ be the smallest positive integer such that $-k \leq q$. Then, we have that
			\begin{equation}
							\label{eq:quasiasympbound0}
				f(\varepsilon x) = \bigoh{\varepsilon^{q} L(\varepsilon)}  \qquad \text{ as } \varepsilon \rightarrow 0^{+} \text{ in } \uddspacereal{\beurou}(\mathbb{R}) 
			\end{equation}
	holds	if and only if there exist $x_0>0$ and continuous functions $F$ and $f_{m}$ on $[-x_0, x_0] \setminus \{0\}$, $m \geq k$, such that
			\begin{equation*}
				f(x) = F^{(k)} + \sum_{m = k}^{\infty} f_{m}^{(m)} , \qquad \text{ on } (-x_0, x_0) ,
			\end{equation*}
		for some $\ell > 0$ (for any $\ell > 0$) there exists $C_{\ell} > 0$ such that
			\begin{equation*}
				\left| f_{m}(x) \right| \leq C_{\ell} \frac{\ell^{m}}{M_{m}} |x|^{q+ m} L(|x|) , \qquad 0 < |x| \leq x_0 ,
			\end{equation*}
		for all $m \geq k$, and $F=0$ when $q>-k$ while if $q=-k$ the function $F$ satisfies, for each $a>0$, the bounds	
			\begin{equation*}
				F(ax) - F(-x) = O_{a}(L(x)),  \qquad x\to0^{+}. 
			\end{equation*}

	\end{theorem}
	
We end this section with a remark that briefly indicates further properties of quasiasymptotic boundedness.

	\begin{remark} Let $\mathcal{Z}^{\beurou}(\mathbb{R})$ be the space of ultradifferentiable functions introduced in   \cite[Section 5]{structquasiultra}. Similarly as in the quoted article, one can show under the assumptions $(M.1)$, $(M.2)'$ and $(M.3)'$:
	\begin{enumerate}[(i)]
	\item If \eqref{eq:quasiasympboundinfinity} holds with an $O$-regularly varying function at infinity $L$ , then $f\in\mathcal{Z}'^{\beurou}(\mathbb{R})$ and the quasiasymptotic boundedness relation \eqref{eq:quasiasympboundinfinity} actually holds true in  $\mathcal{Z}'^{\beurou}(\mathbb{R})$. 
	
	From here one derives the following characterization of $\mathcal{Z}'^{\beurou}(\mathbb{R})$. An ultradistribution $f\in\mathcal{D}'^{\beurou}(\mathbb{R})$ belongs to $\mathcal{Z}'^{\beurou}(\mathbb{R})$ if and only if there is some $q\in\mathbb{R}$ such that $f(\lambda x)=O(\lambda^{q})$ as $\lambda \to \infty$ in $\mathcal{D}'^{\beurou}(\mathbb{R})$. We leave the verification of the direct implication to the reader.
	\item If $f\in\mathcal{Z}'^{\beurou}(\mathbb{R})$ and \eqref{eq:quasiasympbound0} holds with  an $O$-regularly varying function at the origin $L$, then \eqref{eq:quasiasympbound0} is actually valid in  $\mathcal{Z}'^{\beurou}(\mathbb{R})$.
	\item From these two properties, the Banach-Steinhaus theorem, and the fact that $\mathcal{D}^{\ast}(\mathbb{R})$ is dense in $\mathcal{Z}^{\ast}(\mathbb{R})$,  one concludes that \cite[Theorem 5.1 and Theorem 5.3]{structquasiultra} still hold true if one replaces $(M.2)$ and $(M.3)$ there by the weaker assumptions $(M.2)'$ and $(M.3)'$.
	\end{enumerate}
	\end{remark}

\section{The moment asymptotic expansion}
\label{Section MAE}
This section is devoted to study the moment asymptotic expansion (\ref{eq:MAEasymp}), which in general we interpret in the sense of the following definition.

	\begin{definition} Let $\mathcal{X}$ be a l.c.s. of smooth functions provided with continuous actions of the dilation operators and the Dirac delta and all its partial derivatives.
		An element $f \in \mathcal{X}^{\prime}$ is said to satisfy the \textit{moment asymptotic expansion (MAE)} in $\mathcal{X}^{\prime}$ if there are $\mu_{\alpha} \in \complexes$, $\alpha \in \naturals^{d}$, called its \textit{moments}, such that for any $\varphi \in \mathcal{X}$ and $k \in \naturals$ we have
			\begin{equation}
				\label{eq:MAE}
				\ev{f(\lambda x)}{\varphi(x)} = \sum_{|\alpha| < k} \frac{\mu_{\alpha} \varphi^{(\alpha)}(0)}{\alpha! \lambda^{|\alpha| + d}} + \bigoh{\frac{1}{\lambda^{k + d}}} , \qquad \lambda \rightarrow \infty .
			\end{equation}
	\end{definition}

Similarly as in the case of compactly supported distributions \cite{E-K-AsympExpansions,estrada2012distributional} or analytic functionals \cite{Schmidt2005}, one can show that any compactly supported distribution satisfies the MAE in $\mathcal{E}'^{\ast}(\mathbb{R}^{d})$ (we will actually state a stronger result in Proposition \ref{t:E'*hasUMAE} below).
Naturally, as in the distribution case, we expect the MAE to be also valid in larger ultradistribution spaces. In dimension 1, Estrada gave in \cite[Theorem 7.1]{Estrada 1998} (cf. \cite{estrada2012distributional}) a full characterization of the largest distribution space where the moment asymptotic expansion holds; in fact, he showed that $f\in\mathcal{D}'(\mathbb{R})$ satisfies the MAE (in $\mathcal{D}'(\mathbb{R})$) if and only if  $f\in\mathcal{K}'(\mathbb{R})$ (and the MAE holds in this space), where $\mathcal{K}'(\mathbb{R})$ is the dual of the so-called space of  GLS symbols of pseudodifferential operators \cite{GLS1968}. One of our  goals here is to give an ultradistributional counterpart of Estrada's result.

We start by introducing an ultradistributional version of $\mathcal{K}(\mathbb{R}^{d})$.
For each $q \in \naturals$ and $\ell > 0$ we denote by $\mathcal{K}^{M_{p}, \ell}_{q}(\reals^{d})$ the Banach space of all smooth functions $\varphi$ for which the norm
	\begin{equation*}
		\norm{\varphi}_{\mathcal{K}^{M_{p}, \ell}_{q}(\mathbb{R}^{d})} = \sup_{\alpha \in \naturals^{d}} \sup_{x \in \reals^{d}} \frac{(1 + |x|)^{|\alpha| - q} |\varphi^{(\alpha)}(x)|}{\ell^{|\alpha|} M_{\alpha}}
	\end{equation*}
is finite. From this we construct the spaces
	\[ \mathcal{K}^{(M_{p})}_{q}(\reals^{d}) = \varprojlim_{\ell \rightarrow 0^{+}} \mathcal{K}^{M_{p}, \ell}_{q}(\reals^{d}) , \qquad \mathcal{K}^{\{M_{p}\}}_{q}(\reals^{d}) = \varinjlim_{\ell \rightarrow \infty} \mathcal{K}^{M_{p}, \ell}_{q}(\reals^{d}) , \]
and finally the test function space
	\[ \mathcal{K}^{\beurou}(\reals^{d}) = \varinjlim_{q \in \naturals} \mathcal{K}^{\beurou}_{q}(\reals^{d}) . \]
It should be noticed that this is space is never trivial; in fact, $\mathcal{K}^{\beurou}(\reals^{d})$ contains the space of polynomials.	
	
	Our first important result in this subsection asserts that the elements of $\mathcal{K}^{\prime \beurou}(\reals^{d})$ automatically satisfy the MAE. Interestingly, no restriction on the weight sequence $M_p$ is needed to achieve this.
	
	\begin{theorem}
		\label{t:K'*haveMAE}
		Any element $f \in \mathcal{K}^{\prime \beurou}(\reals^{d})$ satisfies the MAE in $\mathcal{K}^{\prime \beurou}(\reals^{d})$ and its moments are exactly $\mu_{\alpha} = \ev{f(x)}{x^{\alpha}}$, $\alpha \in \naturals^{d}$. 
	\end{theorem}
	
	\begin{proof}Let $f \in \mathcal{K}^{\prime \beurou}(\reals^{d})$ We keep $\lambda\geq 1$ and fix $k \in \naturals$. Take any arbitrary $\varphi \in \mathcal{K}^{\beurou}_{q}(\reals^{d})$, where we may assume $q\geq k$.  Consider the $(k-1)$th order Taylor polynomial of $\varphi$ at the origin, that is, $\varphi_{k}(x) := \sum_{|\alpha| < k} \varphi^{(\alpha)}(0) x^{\alpha}/\alpha!.$ Since $\varphi_{k} \in \mathcal{K}^{\beurou}(\reals^{d})$,
	\[
				\ev{f(\lambda x)}{\varphi(x)} = \sum_{|\alpha| < k} \frac{\mu_{\alpha} \varphi^{(\alpha)}(0)}{\alpha! \lambda^{|\alpha| + d}} + \ev{f(\lambda x)}{\varphi(x) - \varphi_{k}(x)} .
	\]
Thus, we need to show $\ev{f(\lambda x)}{\varphi(x) - \varphi_{k}(x)}=O(1/\lambda^{k + d})$. This bound does not require any uniformity in $k$; therefore, we may just assume that $\varphi^{(\alpha)}(0) = 0$ for any $|\alpha| < k$ so that our problem reduces to estimate $|\ev{f(\lambda x)}{\varphi(x)}|$. There exists some $\ell = \ell_{f} > 0$ (some $\ell = \ell_{\phi} > 0$) such that $\varphi\in \mathcal{K}_{q}^{M_p,\ell}(\mathbb{R}^{d})$ and 
			\[ \left| \ev{f(\lambda x)}{\varphi(x)} \right| \leq \frac{\norm{f}_{(\mathcal{K}^{M_{p}, \ell}_{q}(\reals^{d}))^{\prime}}}{\lambda^{d}} \sup_{\alpha \in \naturals^{d}} \sup_{x \in \reals^{d}} \frac{(1 + |x|)^{|\alpha| - q} |\varphi^{(\alpha)}(x / \lambda)|}{\lambda^{|\alpha|} \ell^{|\alpha|} M_{\alpha}} . \]
			
			If $|\alpha|\geq q $, we have
			
	\begin{align*} 
		\sup_{x\in \mathbb{R}^{d}} \frac{(1 + |x|)^{|\alpha| - q} |\varphi^{(\alpha)}(x / \lambda)|}{\lambda^{|\alpha|} \ell^{|\alpha|} M_{\alpha}}
				 &= \frac{1}{\lambda^{q}}\sup_{} \left(\frac{1 + |x|}{\lambda + |x|}\right)^{|\alpha| - q} \frac{(1 + |x|/\lambda)^{|\alpha| - q} |\varphi^{(\alpha)}(x / \lambda)|}{ \ell^{|\alpha|} M_{\alpha}} \\
				&\leq  \frac{\norm{\varphi}_{\mathcal{K}^{M_{p}, \ell}_{q}(\mathbb{R}^{d})} }{\lambda^{k}}.
			\end{align*}
We further consider $|\alpha|<q$. When $|x| \geq \lambda$, obviously 
			\[ \frac{1}{2} \leq \frac{1 + |x|}{\lambda + |x| } \]
	and we obtain
	\[
		\sup_{|x| \geq \lambda} \frac{(1 + |x|)^{|\alpha| - q} |\varphi^{(\alpha)}(x / \lambda)|}{\lambda^{|\alpha|} \ell ^{|\alpha|} M_{\alpha}}
				\leq 2^{q} \frac{\norm{\varphi}_{\mathcal{K}^{M_{p}, \ell}_{q}(\mathbb{R}^{d})} }{\lambda^{k}}.
	\]
		We are left with the case $|x| \leq \lambda$ and $|\alpha|< q$. If $k\leq |\alpha|<q$ we get
			\begin{align*} 
				\sup_{|x| \leq \lambda} \frac{(1 + |x|)^{|\alpha| - q} |\varphi^{(\alpha)}(x / \lambda)|}{\lambda^{|\alpha|} \ell^{|\alpha|} M_{\alpha}} &\leq \frac{1}{\lambda^{k}}\sup_{|x| \leq \lambda} \frac{ |\varphi^{(\alpha)}(x / \lambda)|}{ \ell^{|\alpha|} M_{\alpha}} 
\leq  2^{q-k}\frac{\norm{\varphi}_{\mathcal{K}^{M_{p}, \ell}(
\mathbb{R}^{d})}}{\lambda^{k}} . 
			\end{align*}
	Finally, for $|\alpha| < k$, the Taylor formula yields
			\begin{align*} 
				\sup_{|x| \leq \lambda} \frac{(1 + |x|)^{|\alpha| - q} |\varphi^{(\alpha)}(x / \lambda)|}{\lambda^{|\alpha|} \ell^{|\alpha|} M_{\alpha}}
				&
				 \leq \sup_{|x| \leq \lambda} \frac{(1 + |x|)^{|\alpha| - q}}{\lambda^{|\alpha|} \ell^{|\alpha|} M_{\alpha}} \underset{|\beta|=k}{\sum_{\alpha \leq \beta}} \frac{|\varphi^{(\beta)}(\xi_{x / \lambda})|}{(\beta - \alpha)!} \frac{|x|^{|\beta - \alpha|}}{\lambda^{|\beta - \alpha|}} 
				 \\
				 &
				\leq 2^{q} \frac{C_{\ell, k}}{\lambda^{k}} \norm{\varphi}_{\mathcal{K}^{M_{p}, \ell}_{p}(\mathbb{R}^{d})}.
			\end{align*}
The proof is now complete.
	\end{proof}
	
The next proposition describes the structure of the elements of $\mathcal{K}'^{\ast}(\mathbb{R}^{d})$. The proof in the Beurling case is standard, while in the Roumieu case it can be established via the dual Mittag-Leffler lemma in a similar fashion as in \cite[Section~8]{ultradistributions1}, we therefore leave details to the reader. We point out that the converse of  Proposition \ref{p:structureK*} holds unconditionally, that is, without having to impose any assumption on $M_p$.

	\begin{proposition}
		\label{p:structureK*}
		Let $M_{p}$ satisfy $(M.1)$ and $(M.2)'$. Let $f\in\mathcal{K}'^{\ast}(\mathbb{R}^{d})$. Then, given any $q \in \naturals$ one can find a multi-sequence of continuous functions $f_{\alpha}=f_{q, \alpha}\in C(\mathbb{R}^{d})$ such that 
					\begin{equation}
						\label{eq:MAEstruct}
						f = \sum_{\alpha \in \naturals^{d}} f_{ \alpha}^{(\alpha)}
					\end{equation}
				and for some $\ell > 0$ (for any $\ell > 0$) there is $C = C_{q,\ell} > 0$ such that
					\begin{equation}
						\label{eq:MAEstructbound}
						\left| f_{\alpha}(x) \right| \leq C\frac{\ell^{|\alpha|}}{M_{\alpha}}(1 + |x|)^{|\alpha| - q} , \qquad x \in \reals^{d} ,\ \alpha\in\mathbb{N}^{d}.
					\end{equation}			
	\end{proposition} 

Notice that when $(M.1)$ and $(M.3)'$ hold, then one has the continuous and dense inclusions
$\udspacereals{\beurou}{d} \hookrightarrow \mathcal{K}^{\beurou}(\reals^{d}) \hookrightarrow \uespacereals{\beurou}{d} ,$ so that in particular $\mathcal{K}'^{\beurou}(\reals^{d})\subset \mathcal{D}'^{\ast} (\mathbb{R}^{d})$.
Upon combining Proposition \ref{p:quasiasympboundimposesstruct}(i) with Theorem \ref{t:K'*haveMAE}, one obtains the following complete characterization of those one-dimensional ultradistributions $f\in\mathcal{D}'^{\ast}(\mathbb{R})$ satisfying the MAE:

	\begin{theorem}
		\label{t:MAEstruct}
		Suppose $M_{p}$ satisfies $(M.1)$, $(M.2)'$, and $(M.3)'$. An ultradistribution $f \in \uddspacereal{\beurou}(\mathbb{R})$ satisfies the MAE in $\uddspacereal{\beurou}(\mathbb{R})$ if and only if $f \in \mathcal{K}^{\prime \beurou}(\mathbb{R})$.
	\end{theorem}
	\begin{proof} If $f$ satisfies the MAE, then in particular $f(\lambda x)=O(\lambda^{-q})$ in $\mathcal{D}'^{\ast}(\mathbb{R}\setminus\{0\})$ for each $q\in\mathbb{N}$. Hence, for a fixed but arbitrary $q\in\mathbb{N}$, using Proposition \ref{p:quasiasympboundimposesstruct}(i) and Komatsu's first structural theorem in the case of compactly supported ultradistributions, we can write $f=\sum_{m=1}^{\infty}f^{(m)}_{m}$ in $\mathcal{D}'^{\ast}(\mathbb{R})$ with $f_{m}=f_{q,m}\in C(\mathbb{R})$ such that for some (for each) $\ell>0$ they fulfill bounds $f_{m}(x)=O_{q,\ell}( \ell^{m}(|x|+1)^{m-q-2} /M_m)$. Clearly, this representation yields $f\in\mathcal{K}'^{\ast}_{q}(\mathbb{R})$. Since $q$ was arbitrary, we conclude that $f\in\mathcal{K}'^{\ast}(\mathbb{R})$. For the converse, Theorem \ref{t:K'*haveMAE} shows that a stronger conclusion actually holds.
	\end{proof}
\begin{remark} In dimension $d=1$, this argument gives an alternative way for proving Proposition  \ref{p:structureK*} in the non-quasianalytic case without having to resort in the dual Mittag-Leffler lemma.
\end{remark}

\section{The uniform MAE}\label{section UMAE}

The bound in (\ref{eq:MAE}) is not uniform in general, but in the ultradistributional case it is natural to expect that some sort of uniformity could be present. For instance, we see  below in Proposition \ref{t:E'*hasUMAE} that this is the case for compactly supported ultradistributions. Let us introduce the following uniform variant of the MAE. Throughout this section we work with three weight sequences $M_p$, $N_p$, and $A_p$, and simultaneously denote in short their Beurling or Roumieu cases by $\ast$, $\dagger$, and $\sharp$, respectively. The associated function \cite{ultradistributions1} of $N_{p}$ is given by 
	\[ N(t) := \sup_{p \in \naturals} \log \frac{t^{p} N_{0}}{N_{p}} , \qquad t \geq 0 . \]

\begin{definition} Let $A_{p}$ be a weight sequence and let $\mathcal{X}$ be a l.c.s. of smooth functions provided with continuous actions of the dilation operators and the Dirac delta and all its partial derivatives. An element $f \in \mathcal{X}^{\prime}$ satisfies the \textit{uniform moment asymptotic expansion (UMAE) in $\mathcal{X}^{\prime}$ with respect to $\sharp$} if there are $\mu_{\alpha} \in \complexes$, $\alpha \in \naturals^{d}$, such that for any $\varphi \in \mathcal{X}$ and each $\ell > 0$ (for some $\ell=\ell_{\varphi} > 0$)  the asymptotic formula
			\begin{equation}
				\label{eq:UMAE}
				\ev{f(\lambda x)}{\varphi(x)} = \sum_{|\alpha| < k} \frac{\mu_{\alpha} \varphi^{(\alpha)}(0)}{\alpha! \lambda^{|\alpha| + d}} + \bigoh{\frac{\ell^{k} A_{k}}{\lambda^{k + d}}} , \qquad \lambda \rightarrow \infty, 
			\end{equation}
		holds uniformly for $k \in \naturals$.
	\end{definition}

We now introduce ultradistribution spaces that are closely related to the UMAE. Given $q, \ell > 0$ we denote by $\mathcal{K}^{M_{p}, \ell}_{N_p, q}(\reals^{d})$ the Banach space of all $\varphi \in C^{\infty}(\reals^{d})$ for which
	\begin{equation}
		\norm{\varphi}_{\mathcal{K}^{M_{p}, \ell}_{N_p, q}(\mathbb{R}^{d})} := \sup_{\alpha \in \naturals^{d}} \sup_{x \in \reals^{d}} \frac{e^{-N(q |x|)} (1 + |x|)^{|\alpha|} \left| \varphi^{(\alpha)}(x) \right|}{\ell ^{|\alpha|} M_{\alpha}}
	\end{equation}
is finite. We then define
	\[ \mathcal{K}^{(M_{p})}_{(N_{p})}(\reals^{d}) = \varinjlim_{q \rightarrow \infty} \varprojlim_{\ell  \rightarrow 0^{+}} \mathcal{K}^{M_{p}, \ell}_{N_{p}, q}(\reals^{d}) , \qquad \mathcal{K}^{\{M_{p}\}}_{\{N_{p}\}}(\reals^{d}) = \varinjlim_{\ell \rightarrow \infty}\varprojlim_{q \rightarrow 0^{+}}  \mathcal{K}^{M_{p}, \ell}_{N_{p}, q}(\reals^{d}) ,\]
	and consider the dual $\mathcal{K}'^{\ast}_{\dagger}(\mathbb{R}^{d})$, whose elements satisfy the UMAE as stated in the next theorem.

	\begin{theorem}
		\label{t:Ku'*haveUMAE} Suppose $M_p$ and $N_p$  satisfy  $(M.1)$ and $(M.2)$; in addition we assume that $\inf_{p\in \mathbb{N}} \sqrt[p]{N_{p}}>0$. Set $A_p=N_p \max_{j \leq p} (M_{j}/j!)$. Then, any element $f \in \mathcal{K}^{\prime \ast}_{\dagger}(\reals^{d})$ satisfies the UMAE in $\mathcal{K}^{\prime \beurou}_{\dagger}(\reals^{d})$ w.r.t. $\sharp$.		
	\end{theorem}
	
	\begin{proof} By replacing it by an equivalent sequence, we may assume that $N_p>1$ for each $p$.
		Fix an arbitrary $0 < \varepsilon \leq 1$ in the Beurling case, while we put $\varepsilon = 1$ in the Roumieu case. We will always assume $\lambda \geq H\geq 1$, where $H$ is the parameter in $(M.2)$ (for both sequences). Take any $f \in \mathcal{K}^{\prime \beurou}_{\dagger}(\reals^{d})$ and $\varphi \in \mathcal{K}^{\beurou}_{\dagger}(\reals^{d})$. Arguing as in the proof of Theorem \ref{t:K'*haveMAE}, we need to find a uniform bound for $|\ev{f(\lambda x)}{\varphi(x) - \varphi_{k}(x)}|$, where $\varphi_{k}$ is the $(k-1)$th order Taylor polynomial of $\varphi$ at the origin. There exist $q=q_{\varphi}>0$ and $\ell=\ell_{f} > 0$ ($\ell=\ell_{\varphi} > 0$ and $q=q_{f}>0$) such that $\varphi\in \mathcal{K}^{M_p,\ell}_{N_p,q}(\mathbb{R}^{d})$ and for some $C > 0$
			\[ \left| \ev{f(\lambda x)}{\varphi(x) - \varphi_{k}(x)} \right| \leq \frac{C}{\lambda^{d}} \sup_{\alpha \in \naturals^{d}} \sup_{x \in \reals^{d}} \frac{e^{-N(q |x|)} (1 + |x|)^{|\alpha|} \left| \partial^{\alpha} [\varphi(x / \lambda) - \varphi_{k}(x / \lambda)] \right|}{\ell^{|\alpha|} M_{\alpha}} . \]
We split according to the size of $\alpha \in \naturals^{d}$. 

First suppose that $|\alpha| < k$.  Set $\ell_{0} := \max(1, \ell)$.  From the Taylor expansion and \cite[Proposition 3.6, p. 51]{ultradistributions1} applied to the sequence $N_p$,
			\begin{align*}
				& \frac{(1 + |x|)^{|\alpha|} | \varphi^{(\alpha)}(x/\lambda) - \varphi^{(\alpha)}_{k}(x/\lambda) |}{e^{N(q |x|)}  \lambda^{|\alpha|}  \ell^{|\alpha|} M_{\alpha}} 
				\\
				&
				\leq \frac{e^{-N(q |x|)} (1 + |x|)^{|\alpha|}}{(\lambda \ell)^{|\alpha|} M_{\alpha}} \sum_{\substack{\alpha \leq \beta \\ |\beta| = k}} \frac{|\varphi^{(\beta)}(\xi_{\frac{x}{\lambda}})|}{(\beta - \alpha)!} \left(\frac{|x|}{\lambda}\right)^{|\beta - \alpha|} \\
				&\leq \lambda^{-k} \norm{\varphi}_{\mathcal{K}^{M_{p}, \varepsilon \ell}_{N_{p}, q}(\mathbb{R}^{d})}  \varepsilon^{k} (1+|x|)^{k} e^{N(q |x |/\lambda) - N(q |x|)}  \sum_{\substack{\alpha \leq \beta \\ |\beta| = k}} \frac{\ell^{|\beta - \alpha|} M_{\beta}}{M_{\alpha} (\beta - \alpha)!} \\
				&\leq \lambda^{-k} A N_0\norm{\varphi}_{\mathcal{K}^{M_{p}, \varepsilon \ell}_{N_{p}, q}(\mathbb{R}^{d})} ( d H \ell_{0} \varepsilon)^{k} ( 1+|x|)^{k} e^{-N(q |x| / H) } \sum_{\substack{\alpha \leq \beta \\ |\beta| = k}} \frac{M_{\beta - \alpha}}{|\beta - \alpha|!} \\
				&\leq \lambda^{-k} 2^{d-1}A^2 \norm{\varphi}_{\mathcal{K}^{M_{p}, \varepsilon \ell}_{N_{p}, q}(\mathbb{R}^{d})} \left(4 d q^{-1} H^{2} \ell_0 \varepsilon \right)^{k}  N_{k} \max_{0\leq j\leq k}\frac{M_{j}}{j!}.
			\end{align*}
			
		Now let $|\alpha| \geq k$. For $|x| \geq \lambda$, one has
			\begin{align*} 
				\lambda^{k} \exp[N(q |x| / \lambda)] &= \lambda^{k} \sup_{p \in \naturals} \frac{(q |x| / \lambda)^{p} N_{0}}{N_{p}} \leq \max\left\{ \sup_{p \geq k} \frac{q^{p} |x|^{p} N_{0}}{N_{p}} , \sup_{0 \leq p < k} \frac{q^{p} |x|^{k} N_{0}}{N_{p}} \right\} \\
				&\leq q_{0}^{-k} N_{k} \exp[N(q |x|)] ,
			\end{align*}
		where $q_{0} = \min(1, q)$. Then, since $(1 + |x|)^{|\alpha|} / (1 + |x| / \lambda)^{|\alpha|} \leq \lambda^{|\alpha|}$ for any $\alpha \in \naturals^{d}$, we have
			\[ \sup_{|\alpha| \geq k} \sup_{|x| \geq \lambda} \frac{e^{-N(q |x|)} (1 + |x|)^{|\alpha|} \left| \varphi^{(\alpha)}\left(x / \lambda\right) \right|}{(\lambda \ell)^{|\alpha|} M_{\alpha}} \leq \lambda^{-k} \norm{\varphi}_{\mathcal{K}^{M_{p}, \varepsilon \ell}_{N_{p}, q}(\mathbb{R}^{d})} (\varepsilon / q_{0})^{k} N_{k} . \]
		In the case $|x| \leq \lambda$, we have for $|\alpha| \geq k$,
			\begin{align*}
				\frac{e^{-N(q |x|)} (1 + |x|)^{|\alpha|} \left| \varphi^{(\alpha)}(x / \lambda) \right|}{(\lambda \ell)^{|\alpha|} M_{\alpha}} & \leq \frac{e^{-N(q |x|)} (1 + |x|)^{k} }{\lambda^{k}} \frac{(1 + |x|)^{|\alpha| - k} | \varphi^{(\alpha)}(x / \lambda)|}{\lambda^{|\alpha| - k} \ell^{|\alpha|} M_{\alpha}} \\
				&\leq \lambda^{-k} N^{-1}_{0} e^{N(q)}\norm{\varphi}_{\mathcal{K}^{M_{p}, \varepsilon \ell}_{N_{p}, q}(\mathbb{R}^{d})} (2 \varepsilon / q)^{k} N_{k} ,
			\end{align*} 
			which concludes the proof.

	\end{proof}

The next result describes the UMAE	for compactly supported ultradistributions. The proof goes alone the same lines as that of Theorem \ref{t:Ku'*haveUMAE} and we therefore leave details to the reader.
	
\begin{proposition}
		\label{t:E'*hasUMAE}
		Any element $f \in \udespacereals{\beurou}{d}$ satisfies the UMAE in $\udespacereals{\beurou}{d}$ w.r.t. $\sharp$, where $A_p=\max_{j\leq p}(M_{j}/j!)$. 
	\end{proposition}

A standard argument shows the ensuing structural description for $\mathcal{K}'^{\ast}_{\dagger}(\mathbb{R}^{d})$.

\begin{proposition}
		\label{p:structureK*dagger}
		Let $M_{p}$ satisfy $(M.1)$ and $(M.2)'$. Let $f\in\mathcal{K}'^{\ast}_{}(\mathbb{R}^{d})$. Then, for each $q>0$, there is some $\ell=\ell_{q}$ (for each $\ell$ there some $q_{\ell}>0$) such that one can find a multi-sequence of continuous functions $f_{\alpha}=f_{q,\ell, \alpha}\in C(\mathbb{R}^{d})$ for which $					f = \sum_{\alpha \in \naturals^{d}} f_{\alpha}^{(\alpha)} $	and there is $C = C_{q,\ell} > 0$ such that
					\begin{equation}
						\label{eq:MAEstructbound}
						\left| f_{ \alpha}(x) \right| \leq C \frac{\ell^{|\alpha|}}{M_{\alpha}}(1 + |x|)^{|\alpha|}e^{-N(q |x|)} , \qquad x \in \reals^{d} ,\ \alpha\in\mathbb{N}^{d}.
					\end{equation}		
	\end{proposition}

Let us now consider the one-dimensional case. The ensuing theorem is a counterpart of Theorem \ref{t:MAEstruct} for the UMAE; notice however that a full characterization is lacking in this case.  We mention that if $(M.1)$ and $(M.3)'$ hold, one verifies that $ \udspacereals{\ast}{d} \hookrightarrow \mathcal{K}^{\ast}_{\dagger}(\reals^{d}) \hookrightarrow \uespacereals{\ast}{d} .$

	\begin{theorem}
		\label{t:UMAE1d}
		Suppose that $N_{p}$ satisfies $(M.1)$ and that $(M.1)$, $(M.2)$, and $(M.3) $ hold for the weight sequence $M_p$. Set $A_p= M_p N_p/p!$.  If $f\in\mathcal{D}'^{\ast}(\mathbb{R})$ satisfies the UMAE in $\mathcal{D}'^{\ast}(\mathbb{R})$ with respect to $\dagger$, then  $f \in \mathcal{K}^{\prime \ast}_{\dagger}(\mathbb{R})$ and if in addition $N_p$ satisfies $(M.2)$ and $\inf_{p\in \mathbb{N}} \sqrt[p]{N_{p}}>0$, the UMAE holds for $f$ in $\mathcal{K}^{\prime \ast}_{\dagger}(\mathbb{R})$ w.r.t. $\sharp$.
		\end{theorem}
			\begin{proof}
	It suffices to show that $f \in \mathcal{K}^{\prime \ast}_{\dagger}(\mathbb{R})$. In the Beurling case we take an arbitrary constant sequence $r_{p}=1/q>0$ and in the Roumieu case an arbitrary $(r_{p})\in\mathfrak{R}$. We have that, whenever $\varphi\in \mathcal{D}^{\ast}(\mathbb{R}^{d}\setminus\{0\})$,
\[ \left| \ev{f(\lambda x)}{\varphi(x)} \right| \leq O\left( \frac{ R_{k-1}N_{k-1}}{\lambda^{k+d} N_{0}} \right), \]
which implies, taking infimum over $k$,
\[\left| \ev{f(\lambda x)}{\varphi(x)} \right| = O\left( \lambda^{-(d+1)}\exp\left(-N_{r_{p}}( \lambda)\right)\right) .\]
Appliying Proposition \ref{p:quasiasympboundimposesstruct}(ii), we can write $f=\sum_{m\in\mathbb{N}}f^{(m)}_{m}$ with continuous functions $f_m$ satisfying the bounds 
\[\left| f_{ m}(x) \right| \leq C_{\ell} \frac{\ell^{m}}{M_m}(1 + |x|)^{m-d-1}e^{-N_{r_q}( |x|)} , \qquad x \in \reals ,\ m\in\mathbb{N},\]
for some $\ell>0$ (for each $\ell>0$), where $N_{r_p}$ stands for the associated function of $R_pN_p$. This yields $f \in \mathcal{K}^{\prime \ast}_{\dagger}(\mathbb{R})$ in both cases, as required. (In the Roumieu case we apply \cite[Lemma~4.5(i), p.~417]{D-V-V}.)
 It has been proved by Petzsche \cite[Proposition 1.1]{Petzsche1988} that $(M.3)$ implies the so-called Rudin condition, namely, there is $C$ such that
\[
\max_{j\leq p} \left(\frac{M_{j}}{j!}\right)^{1/j}\leq C \left(\frac{M_{p}}{p!}\right)^{1/p}, \qquad p\in\mathbb{N}
 ;
 \] therefore, the rest follows from Theorem \ref{t:Ku'*haveUMAE}.
	\end{proof}

\end{document}